\documentclass[11pt]{amsart}
\usepackage{graphicx,amscd,color,amsmath,amsfonts,amssymb,amsthm,centernot,hyperref}
\usepackage[initials]{amsrefs}
\usepackage[all]{xy}

\newtheorem{theorem}{Theorem}

\newtheorem{corollary}{Corollary}
\newtheorem{definition}{Definition}

\newtheorem{lemma}{Lemma}
\newtheorem{proposition}{Proposition}

\begin{document}
\title[A general lineability criterion]{A general lineability criterion for complements of vector spaces}

\author[Ara\'{u}jo]{G. Ara\'{u}jo}
\address[G. Ara\'{u}jo]{\mbox{}\newline \indent Departamento de Matem\'{a}tica \newline\indent Universidade Estadual da Para\'{\i}ba \newline\indent 58.429-500 Campina Grande, Brazil.}
\email{gustavoaraujo@servidor.uepb.edu.br}

\author[Barbosa]{A. Barbosa}
\address[A. Barbosa]{\mbox{}\newline \indent Departamento de Matemática \newline \indent Universidade Federal da Paraíba \newline\indent 58.051-900 João Pessoa, Brazil.}
\email{afsb@academico.ufpb.br}

\begin{abstract}
In 1931, Banach proved that, far from being exceptional objects, the Weierstrass functions form a residual set in the space $\mathcal{C}[0,1]$ of continuous functions. Later on, in 1966, V. I. Gurariy showed that, except for zero, there is an infinite-dimensional linear subspace of Weierstrass functions. This was the first example of \textit{lineability}. Over the last decade, this topic has attracted the continuous attention of the mathematical community, with a steady stream of papers being published, many of them in highly ranked mathematical journals. Several lineability criteria are known and applied to specific topological vector spaces. To paraphrase  L. Bernal-Gonz\'alez and M. O. Cabrera in [J. Funct. Anal. \textbf{266} (2014), 3997-4025], ``sometimes, such criteria furnish unified proofs of a number of scattered results in the related literature''. In this article, we provide a general lineability criterion in the context of complements of vector spaces.
\end{abstract}


\thanks{The first author was supported by Grant 3024/2021, Para\'iba State Research Foundation (FAPESQ). The second author was supported by CAPES;}

\thanks{\textbf{Conflict of interest.} The authors declare that they have no conflict of interest.}

\maketitle


\section{Preliminaries and background}\label{sec1}

From now on all vector spaces are considered over a fixed scalar field $\mathbb{K}$ which can be either $\mathbb{R}$ or $\mathbb{C}$. For any set $X$ we shall denote by $\mathrm{card}(X)$ the cardinality of $X$; we also define $\mathfrak{c} = \mathrm{card}(\mathbb{R})$ and $\aleph_0 = \mathrm{card} (\mathbb{N})$. If $X$ is a vector space and $Y$ is a vector subspace of $X$, we denote the infinite algebraic codimension of $Y$ by $\mathrm{codim}(Y)$.

Recall that if $E$ is a vector space, $\beta\leq\dim(E)$ is a cardinal number and $A \subset E$, then $A$ is said to be:
\begin{enumerate}
\item[$\bullet$] {\it lineable} if there is an infinite dimensional vector space $F$ such that $F \smallsetminus \{0\} \subset A$,
\item[$\bullet$] {\it $\beta$-lineable} if there exists a vector space $F_\beta$ with $\dim(F_\beta) = \beta$ and $F_\beta \smallsetminus \{0\} \subset A$ (hence lineability means $\aleph_0$-lineability),
\item[$\bullet$] {\it maximal lineable} in $E$ if $A$ is $\dim(E)$-lineable, and
\item[$\bullet$] {\it pointwise $\beta$-lineable} if for each $x\in A$ there is a subspace $F_{\beta,x}$ such that $x \in F_{\beta,x}\subset A\cup \{0\}$ and $\dim(F_{\beta,x}) = \beta$.
\end{enumerate}
Also, if $\alpha$ is another cardinal number, with $\alpha<\beta$,  then $A$ is said to be:
\begin{enumerate}
\item[$\bullet$] {\it $(\alpha,\beta)$-lineable} if it is $\alpha$-lineable and for every subspace $F_\alpha\subset E$ with $F_\alpha\subset A\cup\{0\}$ and $\dim(F_\alpha)=\alpha$, there is subspace $F_\beta\subset E$ with $\dim (F_\beta)=\beta$ and
\begin{equation}\label{ab}
F_\alpha\subset F_\beta \subset A\cup \{0\}
\end{equation}
(hence $(0,\beta)$-lineability $\Leftrightarrow \ \beta$-lineability).
\end{enumerate}

The classical concept of \textit{lineability} (the first three points above) was coined by V. I. Gurariy in the early 2000's and it first appeared in print in \cite{GQ,AGS,SS}. V. I. Gurariy's interest in linear structures in generally non-linear settings dates as far back as 1966 (see \cite{gurariy1966}). The study of large vector structures in sets of real and complex functions has attracted many mathematicians in the last decade. For example, in \cite[Theorem 4.3]{AGS} (see also \cite{10}), the authors prove that the set of everywhere surjective functions on $\mathbb{R}$ is $2^{\mathfrak{c}}$-lineable. This result has been further improved in \cite{GMSS}, where the authors prove that the set of strongly everywhere surjective functions and the set of perfectly everywhere surjective functions are both $2^{\mathfrak{c}}$-lineable as well. In the same paper, the authors show that the set of nowhere monotone everywhere differentiable functions on $\mathbb{R}$ is $\mathfrak{c}$-lineable (see also \cite{2}).

Situations like those described above often occur, that is, in a vector space, lineability results for one of its subsets with certain apparently rare properties are generally positive. In fact, these sets often contain a large-dimensional vector space. In this sense, more restrictive variations of the concept of lineability (and its variants) were created in order to obtain more non-lineability results.

The notion of pointwise lineability was introduced in \cite{PR} and the notion of $(\alpha,\beta)$-lineability was introduced in \cite{FPT} (see \cite{DR,DFPR,FPR}). Initially, our aim in this study was to obtain necessary (and possibly sufficient) conditions to achieve results of non-$(\alpha,\beta)$-lineability in the context of complements of vector spaces. Contrary to our initial objective, what we achieved was, in fact, a very general $(\alpha,\beta)$-lineability criterion (and therefore also a criterion of lineability) in the context of complements of vector spaces. As a consequence of our main result (Theorem \ref{T1}), we can verify, for instance, that $E\smallsetminus F$ is always $\beta$-lineable, where $E$ is any non-trivial vector space, $F$ is a proper subspace of $E$ and $\beta=\mathrm{codim}(F)$. That is, in $E\smallsetminus F$, the concepts of $\beta$-lineability, $(\alpha,\beta)$-lineability and pointwise $\beta$-lineability are actually the same.

The paper is organized as follows. In the second section we present our first main result, a very generic criterion of $(\alpha,\beta)$-lienability in the context of differences in vector spaces, and as a consequence, we present several corollaries that can be applied in the most varied contexts. In Section \ref{sec3}, through the study of the set
\begin{equation}\label{AAA}
L_p[0,1]\smallsetminus\bigcup_{q>p}L_q[0,1], \ \ \ p>0,
\end{equation}
we recall that the main theorem of Section \ref{sec2} cannot be improved, in general, for the context of \textit{spaceability} (this definition will be presented in Section \ref{sec3}), because a result by Fávaro et al. \cite{fprs} ensures that the set in \eqref{AAA} is not $(\alpha,\beta)$-spaceable for $\alpha\geq\aleph_0$, regardless of the cardinal number $\beta$. In this section, we give (few and weak) sufficient conditions to achieve general non-$(\alpha,\beta)$-spaceability criteria. In Section \ref{sec4}, as a matter of curiosity, we will detail how to retrieve and generalize, as a consequence of our main results (Theorems \ref{T1}, \ref{T6} and \ref{C7}) and the corollaries derived from it (Corollaries \ref{1}, \ref{S1}, \ref{3} and \ref{JKHSKJAHK}), some lineability/spaceability classical results.

\section{Main result and some consequences}\label{sec2}

Only a few (and very recent) non-constructive techniques strategy to obtain lineability are known. The main ones are due to Aron et al. \cite{agps2009}, Bernal-Gonz\'alez et al. \cite{bc}, and G\'amez et al. \cite{gamezmunozseoane2010,LAA2014}). These techniques (although they represent major breakthroughs in the theory) unfortunately only apply to certain particular frameworks. They are, actually, far from being useful in more general and abstract contexts.

Here we present another very general criterion of lineability.

\begin{theorem}\label{T1}
Let $E$ be a vector space and $F$ be a proper subspace of $E$. Then $E\smallsetminus F$ is $(\alpha,\mathrm{codim}(F))$-lineable for all $\alpha < \mathrm{codim}(F)$.
\end{theorem}

\begin{proof}
Let $\mathcal{A}\subset E$ be a linearly independent subset such that $\mathrm{span}(\mathcal{A})\cap F=\{0\}$ and $\mathrm{card}(\mathcal{A})=\alpha$ and let $\mathcal{B}$ be a base of $F$.

Let us prove that $\mathcal{A}\cup\mathcal{B}$ is a linearly independent set. In fact, given $u_1,\ldots,u_n\in \mathcal{A}$, $v_1,\ldots,v_m\in\mathcal{B}$ and $a_1,\ldots,a_n,b_1,\ldots,b_m\in\mathbb{K}$ such that
\[
\sum_{i=1}^{n}a_iu_i+\sum_{j=1}^{m}b_jv_j=0,
\]
it follows that
\[
\mathrm{span}(\mathcal{A})\ni\sum_{i=1}^{n}a_iu_i=-\sum_{j=1}^{m}b_jv_j\in F.
\]
Since $\mathrm{span}(\mathcal{A})\cap F=\{0\}$ and $\mathcal{A}$ and $\mathcal{B}$ are linearly independent, we conclude that $a_i=b_j=0$, $i=1,\ldots,n$, $j=1,\ldots, m$.

Since every linearly independent set of a vector space can be extended to a basis, we can consider $\mathcal{C}\subset E$ so that $\mathcal{A}\cup\mathcal{B}\cup \mathcal{C }$ is a base of $E$. Therefore, if $u\in\mathrm{span}(\mathcal{A}\cup\mathcal{C})\cap F$, there are $u_1,\ldots,u_n\in\mathcal{A}\cup\mathcal{C}$, $v_1,\ldots,v_m\in\mathcal{B}$ and $a_1,\ldots,a_n,b_1,\ldots,b_m\in\mathbb{K}$ so that
\[
\sum_{i=1}^{n}a_iu_i=u=\sum_{j=1}^{m}b_jv_j,
\]
i.e.,
\[
\sum_{i=1}^{n}a_iu_i+\sum_{j=1}^{m}-b_jv_j=0.
\]
Due to the linear independence of $\mathcal{A}\cup \mathcal{B}\cup \mathcal{C}$ we get $a_i=b_j=0$, $i=1,\ldots,n$, $j=1,\ldots,m$, and thus $u=0$. This shows that
\[
\mathrm{span}(\mathcal{A}\cup\mathcal{C})\cap F=\{0\} \ \ \ \text{and} \ \ \ \mathrm{span}(\mathcal{A})\subset\mathrm{span}(\mathcal{A}\cup \mathcal{C})
\]

Let us now see that the restriction of the canonical projection $E \ni x \stackrel{\pi}{\longmapsto} \overline{x}\in \frac{E}{F}$ to the subspace $V=\mathrm{span}(\mathcal{A}\cup\mathcal{C})$ is a linear isomorphism. In fact, given $\bar{x}\in\frac{E}{F}$, there are $u_x\in V$ and $v_x\in F$ such that $x=u_x+v_x$ and
\[
\pi(u_x) = \overline{u_x} = \overline{u_x}+\overline{0} = \overline{u_x}+\overline{v_x} = \overline{u_x+v_x}=\overline{x},
\]
i.e., $\pi|_V$ is surjective. Now, let $x,y\in V$ such that $\pi(x)=\pi(y)$. Since $\overline{x-y}=\overline{x}-\overline{y}=\overline{0}$, we conclude that $x-y\in F$. As we also have $x-y\in V$ and $V\cap F=\{0\}$, it follows that $x-y=0$, which shows the injectivity of $\pi|_V$.

Since $V$ and $\frac{E}{F}$ are isomorphic, we can conclude that $\dim (V)=\mathrm{codim}(F)$. By construction we also have $\mathrm{span}(\mathcal{A})\subset V\subset (E\smallsetminus F)\cup{0}$, and the result is proved.
\end{proof}  

The next results are immediate consequences of the previous theorem.

\begin{corollary}\label{1}
Let $E$ be a vector space and $F$ be a proper subspace of $E$. Then $E\smallsetminus F$ is $\mathrm{codim}(F)$-lineable.
\end{corollary}

\begin{corollary}\label{S1}
Let $E$ be a vector space of dimension $\beta$ and $F$ be a proper subspace of $E$. If $E\smallsetminus F$ is $\beta$-lineable, then $E\smallsetminus F$ is $(\alpha,\beta)$-lineable for all $\alpha < \beta$.
\end{corollary}

\begin{proof}
Since $E\smallsetminus F$ is $\beta$-lineable, it follows that $\beta\leq \dim \left(\frac{E}{F}\right)$.
Thus $\dim \left(\frac{E}{F}\right)=\beta$. From Theorem \ref{T1} we conclude that $E\smallsetminus F$ is $(\alpha,\beta)$-lineable.
\end{proof}

It is not difficult to verify that
\[
\text{pointwise }\beta\text{-lineability }\Longrightarrow \ (1,\beta)\text{-lineability},
\]
and that the reciprocal is not necessarily true. As a consequence of our main theorem (Theorem \ref{T1}), we can verify that, in the context of the complement of vector spaces, these two notions coincide. In fact, given $x\in E\smallsetminus F$, we have $\mathrm{span}(\{x\})\subset (E\smallsetminus F)\cup\{0\}$ and therefore, by Theorem \ref{T1}, there is a subspace $V$ of $E$, with $\dim(V)=\beta$, such that $\mathrm{span}(\{x\})\subset V\subset (E\smallsetminus F)\cup\{0\}$. In particular $x\in V$.

Therefore, the following result is valid:

\begin{corollary}\label{3}
Let $E$ a vector space and $F$ a proper subspace of $E$. Then $E\smallsetminus F$ is pointwise $\mathrm{codim}(F)$-lineable.
\end{corollary}

\section{What can we say about spaceability and $(\alpha,\beta)$-spaceability?}\label{sec3}

In the definition of lineability given in Section \ref{sec1}, if $E$ is, in addition, a topological vector space, then $A$ is called {\it $\beta$-spaceable} if $A\cup\{0\}$ contains a closed $\beta$-dimensional linear subspace of $E$, and {\it maximal spaceable} in $E$ if $A$ is $\dim(E)$-spaceable (see \cite{AGS}). Moreover, if the subspace $F_\beta$ satisfying \eqref{ab} can always be chosen closed, we say that $A$ is $(\alpha,\beta)$-spaceable (see \cite{FPT}).

Before we comment on this question, recall that, for $p>0$, the $L_p[0,1]$ space is the quotient of the space
\begin{align*}
& \mathcal{L}_p[0,1]\\
& =\left\{f:[0,1]\rightarrow\mathbb{K}:f\text{ is measurable and} \int_{[0,1]}|f(x)|^p \mu(x)<\infty\right\}
\end{align*}
by the equivalence relation
\[
f\sim g \Leftrightarrow \mu(\{x\in[0,1]:\ f(x)\neq g(x)\})=0,
\]
equipped with the norm ($p$-norm if $p<1$)
\begin{eqnarray*}
\|f\|_p=\left(\int_{[0,1]}|f(x)|^p \mu(x)\right)^\frac{1}{p}.
\end{eqnarray*}

In Botelho et al. \cite{BFPS} it was proved that
\begin{equation}\label{lp-lq}
L_p[0,1]\smallsetminus\bigcup_{q>p}L_q[0,1]
\end{equation}
is spaceable for all $p>0$, but the proof does not assure that $L_p[0, 1]\smallsetminus\cup_{q>p}L_q[0, 1]$ is $(\alpha,\mathfrak{c})$-spaceable for some cardinal $\alpha>0$.

A result by V. V. Fávaro et al. \cite{FPT} shows that this is true for $\alpha = 1$ and in the same article they ask about the $(\alpha,\mathfrak{c})$-spaceability of the set in \eqref{lp-lq} for a cardinal $1 <\alpha < \mathfrak{c}$ (this same issue is again highlighted in \cite{FPR}). Later, in \cite[Corollary 2.4]{fprs}, V. V. Fávaro et al. proved that the set in \eqref{lp-lq} is not $(\alpha,\beta)$-spaceable for $\alpha\geq\aleph_0$, regardless of the cardinal number $\beta$. Therefore the question remains open only for $2\leq \alpha<\aleph_0$.

A direct consequence of our main result shows us a different picture when we deal only with the lineability.

\begin{corollary}
Let $p>0$. The set
\[
L_p[0,1]\smallsetminus\bigcup_{q>p}L_q[0,1]
\]
is $(\alpha,\mathfrak{c})$-lineable for all $\alpha<\mathfrak{c}$.
\end{corollary}

In this section, our main objective is to give (few and weak) sufficient conditions to achieve general non-$(\alpha,\beta)$-spaceability criteria. In this sense, let us recall the following definition.

\begin{definition}
Let $E$ be a topological vector space and $\alpha$ be a cardinal number. We say that a subset $A$ of $E$ is $\alpha$-dense-lineable whenever $A \cup \{0\}$ contains a dense vector subspace $F$ of $E$ with $\dim(F) = \alpha$.
\end{definition}

The following lemma will be useful to prove one of the main theorems of this section (Theorems \ref{T6}).

\begin{lemma}\label{L5}
Let $E$ be a topological vector space and $A$ an $\alpha$-dense-lineable subset of $E$. If $A$ is $(\alpha,\beta)$-lineable for some $\alpha<\beta$, then $A$ is $\beta$-dense-lineable.
\end{lemma}

\begin{proof}
Let $F_\alpha$ be a dense vector space of dimension $\alpha$ contained in $A\cup\{0\}$. Since $A$ is $(\alpha,\beta)$-lineable, there exists a subspace $F_\beta$ of $E$, of dimension $\beta$, containing $F_\alpha$ and contained in $A\cup\{0\}$. Since $F_\alpha$ is dense in $E$, it follows that $F_\beta$ is also dense, from which we obtain the $\beta$-dense-lineability of $A$.
\end{proof}

\begin{theorem}\label{T6}
Let $E$ be a topological vector space and $A$ a subset of $E$ with $\mathrm{card}(E\smallsetminus A)\geq2$. If $A$ is $\alpha$-dense-lineable and $(\alpha,\beta)$-lineable, then $A$ is not $(\gamma,\lambda)$-spaceable for any $\alpha\leq\gamma\leq\beta$, regardless of the cardinal number $\lambda$.
\end{theorem}

\begin{proof}
Since $(\alpha,\beta)$-lineability implies $(\alpha,\gamma)$-lineability for all $\gamma\in(\alpha,\beta)$, as a consequence of Lemma \ref{L5} we have that $A$ is $\gamma$-dense-lineable for all $\gamma\in[\alpha,\beta]$. Let $F$ be a dense subspace of dimension $\gamma$ contained in $A\cup\{0\}$. The only closed subspace of $E$ that contains $F$ is $E$. Since $E\not\subset A\cup\{0\}$, we conclude that $A$ cannot be $(\gamma,\lambda)$-spaceable.
\end{proof}

An immediate consequence of the above theorem is the following result:

\begin{corollary}\label{JKHSKJAHK}
Let $E$ be a topological vector space and $A$ a subset of $E$ with $\mathrm{card}(E\smallsetminus A)\geq2$. If $A$ is $\alpha$-dense-lineable, then $A$ is not $(\alpha,\beta)$-spaceable.
\end{corollary}

The demonstration given in \cite{fprs} that $L_p[0,1]\smallsetminus\cup_{q>p}L_q[0,1]$ is not $(\alpha,\beta)$-spaceable for $\alpha\geq\aleph_0$, regardless of the cardinal number $\beta$, was a consequence of the following result:

\begin{corollary}{\cite[Corollary 2.3]{fprs}}\label{AAAAA}
Let $\alpha\geq\aleph_0$ and $\beta$ be a cardinal number. Let $E$ be a Banach space or $p$-Banach space ($p > 0$) and $F$ be a non-trivial subspace of $E$. If $E \smallsetminus F$ is $\alpha$-lineable then $E \smallsetminus F$ is not $(\alpha, \beta)$-spaceable.
\end{corollary}

Our next main result will, in a sense, weaken the assumptions of Corollary \ref{AAAAA}. On the one hand, we will \textit{only} ask that $E$ be a metrizable topological vector space, but, in return, the separability of $E$ will also be required.

We will need the following result:

\begin{theorem}\cite[Theorem 2.5]{bc}\label{T8}
Let $E$ be a metrizable separable topological vector space and $F$ be a vector subspace of $E$. If $F$ has infinite codimension, then $E \smallsetminus F$ is $\aleph_0$-dense-lineable.
\end{theorem}

\begin{theorem}\label{C7}
Let $E$ be a metrizable separable topological vector space and $F$ be a proper vector subspace of $E$. If $F$ has infinite codimension, then $E\smallsetminus F$ is not $(\alpha,\beta)$-spaceable for all $\aleph_0\leq\alpha\leq\mathrm{codim}(F)$.
\end{theorem}

\begin{proof}
From Theorem \ref{T8} we conclude that $E\smallsetminus F$ is $\aleph_0$-dense-lineable. From Theorem \ref{T1} we know that $E\smallsetminus F$ is $(\alpha,\mathrm{codim}(F))$-lineable for all $\alpha<\mathrm{codim}(F)$, and in particular for $\alpha\geq\aleph_0$. The result follows from Theorem \ref{T6}.
\end{proof}

In addition to several other applications in other contexts, the above result provides a new demonstration for the not $(\alpha,\beta)$-spaceability of $L_p[0,1]\smallsetminus\cup_{q>p} L_q[0,1]$ for all $p>0$ and all $\alpha\geq\aleph_0$.


\[
\]







\section{Applications}\label{sec4}

Our main results (Theorems \ref{T1}, \ref{T6} and \ref{C7}) and the corollaries derived from it (Corollaries \ref{1}, \ref{S1}, \ref{3} and \ref{JKHSKJAHK}) recover or generalize many classical lineability results. Among them, we can highlight some results of \cite{AP,bbfp,BCP,BDFP,FPR,DR2,PS,daniel}.

As a matter of curiosity, we will detail how to retrieve and generalize, for example, the results of \cite{bbfp} and \cite{BDFP}. Let us first recall some definitions. Given a normed space $X$, for $p>0$ we denote by $\ell_p(X)$ the space of sequences $(x_k)_{k\in\mathbb{N}}$ with values in $X$ with $\sum_{k}\|x_k\|^p<\infty$, equipped with the natural norm ($p$-norm if $p<1$)
\[
\|(x_k)_{k\in\mathbb{N}}\|_p=\left(\sum_{k}\|x_k\|^p\right)^\frac{1}{p}.
\]
When $X=\mathbb{K}$ we simply denote $\ell_p(\mathbb{K})=\ell_p$.

In \cite[Corollary 3.5(c)]{bbfp} the authors proved that for an infinite-dimensional Banach space $X$ and for $1\leq p<\infty$ the set $$\bigcap_{p<q}\ell_q (X)\smallsetminus\ell_p(X)$$ is maximal spaceable in $\cap_{p<q}\ell_q(X)$. However, it is not difficult to see that the lineability of the above result is still valid for $0<p<\infty$. In fact, consider $\{x_\gamma\in X: \gamma\in \Gamma\}$ a basis of $X$. Note that $$\left(\frac{1}{n^\frac{1}{p}}x_\gamma\right)_{n\in\mathbb{N}}\in\bigcap_{p<q}\ell_q(X)\smallsetminus\ell_p(X)$$ for all $\gamma\in \Gamma$. It is also not difficult to verify that $\left\{\left(\frac{1}{n^\frac{1}{p}}x_\gamma\right)_{n\in\mathbb{N}}:\gamma\in \Gamma\right\}$ is a linearly independent set such that every nontrivial linear combination still remains in $\cap_{p<q}\ell_q(X)\smallsetminus\ell_p(X)$, that is, $\cap_{p<q}\ell_q(X)\smallsetminus\ell_p(X)$ is maximal lineable.

Corollary \ref{S1} together with the aforementioned improvement of the Corollary 3.5(c) gives us the following result:

\begin{proposition}
Let $X$ be an infinite-dimensional Banach space and $\beta=\dim(X)$. Then the set $$\bigcap_{p<q}\ell_q(X)\smallsetminus\ell_p(X)$$ is $(\alpha,\beta)$-lineable for all $\alpha<\beta$ and all $p>0$.
\end{proposition}

\begin{proof}
Since $$\dim \left(\bigcap_{p<q}\ell_q(X)\right)=\dim (\ell_p(X))=\dim (X)=\beta,$$ the $(\alpha,\beta)$-lineability is a consequence of Corollary \ref{S1}.
\end{proof}

In \cite{BDFP} the authors introduce a very general class of sequence spaces and verify that some of its notable subspaces have spaceable complements.

\begin{definition}
Let $X\neq{0}$ be a Banach space.
\begin{itemize}
\item Given $x\in X^{\mathbb{N}}$, we denote by $x^0$ the zero-free version of $x$, i.e., if $x$ has only a finite amount of non-zero terms, then $x^0=0$; otherwise, $x^0=(x_j)_{j=1}^{\infty}$ where $x_j$ denotes the $j$-th non-zero coordinate of $x$.
\item A \textbf{Invariant sequence space on} $X$ is a Banach or quasi-Banach space of sequences with values in $X$ satisfying the following conditions:
\begin{itemize}
\item[a)] For $x\in X^{\mathbb{N}}$ with $x^0\neq0$, $x\in E$ if and only if $x^0\in E$, and in this case $\|x\|\leq K\|x^0\|$ for some constant depending only on the space $E$.
\item[b)] $\|x_j\|_X\leq\|x\|_E$ for all $x\in E$ and every $j\in\mathbb{N}$.
\end{itemize}
\item A \textbf{invariant sequence space} is an invariant sequence space over some Banach space $X$.
\end{itemize}
\end{definition}

The main result of \cite{BDFP} is the following:

\begin{theorem}\label{T7}
Let $E$ be an invariant sequence space over the Banach space $X$. Then
\begin{itemize}
\item[(a)] For each $\Gamma\subset (0,\infty]$, the set $$E\smallsetminus\bigcup_{p\in\Gamma}\ell_p(X)$$ is empty or spaceable.
\item[(b)] $E\smallsetminus c_0(X)$ is empty or spaceable.
\end{itemize}
\end{theorem}

Under some conditions on the spaces $X$ and $E$, we get a $(\alpha,\beta)$-lineability result in this context:

\begin{proposition}\label{T80}
Let $\varnothing\neq\Gamma\subset(0,\infty]$ and let $X$ be a Banach space with $\mathrm{card}(X)=\mathfrak{c}$. If $E$ is an invariant sequence space over $X$, then $$E\smallsetminus\bigcup_{p\in\Gamma}\ell_p(X)$$ is empty or $(\alpha,\mathfrak{c})$-lineable for all $\alpha <\mathfrak{c}$.
\end{proposition}

\begin{proof}
If $$E\smallsetminus\bigcup_{p\in\Gamma}\ell_p(X)$$ is non-empty, by Theorem \ref{T7} we obtain its $\mathfrak{c}$-lineability. Since we have $\dim (E)\geq\mathfrak{c}$ and $$\dim \left(\bigcup_{p\in\Gamma}\ell_p(X)\right)\geq\dim (\ell_p(X) )\geq\mathfrak{c},$$ it follows that
\[
\dim \left(\bigcup_{p\in\Gamma}\ell_p(X)\right)=\mathrm{card}\left(\bigcup_{p\in\Gamma}\ell_p(X)\right) \ \ \ \text{and} \ \ \ \dim (E)=\mathrm{card}(E).
\]
On the other hand, since $$\bigcup_{p\in\Gamma}\ell_p(X), E\subset X^{\mathbb{N}}$$ and $\mathrm{card}(X^{\mathbb{N} })=\mathfrak{c}^{\aleph_0}=\mathfrak{c}$, we conclude that
\[
\dim \left(\bigcup_{p\in\Gamma}\ell_p(X)\right)\leq\mathrm{card}(X^{\mathbb{N}})=\mathfrak{c} \ \ \text{and} \ \ \dim (E)\leq\mathrm{card}(X^{\mathbb{N}})=\mathfrak{c}.
\]
Therefore $$\dim \left(\bigcup_{p\in\Gamma}\ell_p(X)\right)=\dim (E)=\mathfrak{c},$$ and thus we are in the conditions of Corollary \ref{S1}, which gives us, for all $\alpha<\mathfrak{c}$, the $(\alpha,\mathfrak{c})$-lineability of $$E\smallsetminus\bigcup_{p\in\Gamma}\ell_p(X).$$
\end{proof}

\begin{corollary}
The set $$c_0\smallsetminus\bigcup_{p>0}\ell_p$$ is $(\alpha,\mathfrak{c})$-lineable for all $\alpha<\mathfrak{c}$
\end{corollary}

\begin{proof}
To prove the result, it suffices to verify that $$c_0\smallsetminus\bigcup_{p>0}\ell_p\neq\emptyset,$$ because with this fact the Proposition \ref{T80} assures us of its $(\alpha,\mathfrak{c })$-lineability.

Let $(x_k)_{k\in\mathbb{N}}$ be given by
\[
x_1=1 \ \ \ \text{and} \ \ \ x_k=\frac{1}{\log k}, \ k>1.
\]
It is obvious that $(x_k)_ {k\in\mathbb{N}}\in c_0$. Since $\sum_{i=1}^{\infty}\frac{1}{i}=\infty$, to verify that $(x_k)_{k\in\mathbb{N}}\not \in\ell_p$, $p\in\mathbb{N}$, it suffices to show that $$\frac{1}{(\log k)^p}\geq\frac{1}{k}$$ for $k$ large enough. In fact, observe that
\begin{eqnarray*}
\lim_{k\rightarrow\infty}\frac{(\log k)^p}{k}=\lim_{k\rightarrow\infty}\frac{p(\log k)^{p-1}}{k}=\cdots=\lim_{k\rightarrow\infty}\frac{p!\log k}{k}=\lim_{k\rightarrow\infty}\frac{p!}{k}=0
\end{eqnarray*}
and thus we have $\frac{(\log k)^p}{k}\leq 1$ for $k$ large enough.

Since $p\in\mathbb{N}$ was arbitrary, it follows that $$(x_k)_{k\in\mathbb{N}}\in c_0\smallsetminus\bigcup_{p\in\mathbb{N}}\ell_p=c_0\smallsetminus\bigcup_{p>0}\ell_p.$$
\end{proof}

\end{document}